\documentclass[a4paper,10pt]{article}
\usepackage[utf8]{inputenc}
\usepackage{amsmath}
\usepackage{amsthm}
\usepackage{amsfonts}
\usepackage{amssymb}
\usepackage{mathrsfs}
\usepackage{color}
\usepackage{geometry}
\usepackage{caption}
\usepackage{comment}
\usepackage[]{algorithm2e}
\title{Combinatorial proof of the transcendence of $L(1,\chi_s)/\Pi$}
\author{YINING HU \\
CNRS, Institut de Math\'ematiques de Jussieu-PRG \\
Universit\'e Pierre et Marie Curie, Case 247 \\
4 Place Jussieu \\
F-75252 Paris Cedex 05 (France) \\
{\tt yining.hu@imj-prg.fr}}

\date{}
\begin{document}

\maketitle

\newtheorem{thm}{Theorem}
\newtheorem{lem}{Lemma}
\newtheorem{prop}{Proposition}
\newtheorem{coro}{Corollary}
\theoremstyle{definition}
\newtheorem{defi}{Definition}

\begin{abstract}
We give a combinatorial proof of the transcendence of $L(1,\chi_s)/\Pi$, where $L(1,\chi_s)$ (resp. $\Pi$) is the analogue in characteristic
$p$ of the function $L$ of Dirichlet (resp. $\pi$). This result has been proven by G. Damamme using the criteria of de Mathan.
Our proof is based on the Theorem of Christol and another property of $k$-automatic sequences.
\end{abstract}

\section{Introduction}

$$[k]=T^{q^k}-T,$$
$$L_k=[k]...[1], \;\;L_0=1.$$
$$\Pi=\prod\limits_{j=1}^{\infty} \left(1-\frac{[j]}{[j+1]}\right)$$

\begin{thm}[Theorem 2 in \cite{damamme}]\label{dm}
 For $s<q$ and  $a\in\mathbb{F}_q $, 
 $$L(1,\chi_s)=\sum_{k=0}^{\infty}(-1)^{k(s-1)}\frac{(T-a)^{s\frac{q^k-1}{q-1}}}{L_k}.$$
\end{thm}

The following Theorem is proved in \cite{damamme} as a corollary of Theorem \ref{dm} using the criteria of De Mathan.

\begin{thm}[Corollary 2 in \cite{damamme}]\label{main}
 For $1<s<q$, $L(1,\chi_s)/\Pi$ is transcendental over $\mathbb{F}_q(T)$.
\end{thm}

Our goal in this article is to give another proof of Theorem \ref{main} starting from the expression of Theorem \ref{dm}, 
by means of properties of automatic sequences.

For an integer $k\geq 2$, one of the equivalent definitions of
a $k$-automatic sequence is a sequence that can be generated by a $k$-DFAO (deterministic finite automaton 
with output). We recall here the definition of the latter as we will need it in the proof of Lemma \ref{finite}:\\
A $k$-DFAO is a $6$-tuple 
$$M=(Q,\Sigma_k,\delta,q_0,\Delta,\tau)$$
where 
$Q$ is a finite set of states,
$\Sigma_k$ the input alphabet $ \{0,1,...,k-1\},$
$\delta: Q\times \Sigma_k \rightarrow Q $ the transition function,
$q_0\in Q$ the initial state,
$\Delta$ the output alphabet, and $\tau: Q\rightarrow \Delta$ the output function. We expand $\delta$ to a function from 
$Q\times \Sigma_k\rightarrow Q$ by defining, for a word $w=w_1...w_j$ of length at least $2$ in $\Sigma_k^*$, 
$\delta(q,w)=\delta(\delta(1,w_j),w_1...w_{j-1})$. 
The sequence $(u(n))_{n\geq 0}$ generated by the automaton $M$ is defined by $u(0)=\tau(q_0)$ and $u(n)=\tau(\delta(q_0,(n)_k))$ for $n>0$, 
where $(n)_k$ is the base-$k$ expansion
of $n$. In other words, we define $u(n)$ to be the output when we feed the base-$k$ expansion of $n$ to $M$ starting from the least significant digit. 

The following theorem reduces the problem of proving the transcendence of a series over $\mathbb{F}_q(T)$
to proving the non-$q$-automaticity of the sequence of its
coefficients.

\begin{thm}[Christol, Kamae, Mend\`es France, and Rauzy]\label{Christol}
 The formal power series $f(T)=\sum_{n\geq 0}^{\infty} f_n T^{-n} \in \mathbb{F}_q\left[\left[\frac{1}{T}\right]\right]$ is algebraic
 over the fraction field $\mathbb{F}_q(T)$
 if and only if the sequence $(f_n)_n$ is $q$-automatic.
\end{thm}

The following lemma gives a necessary condition of $k$-automaticity, and therefore a way of proving that a sequence is not $k$-automatic.
For a letter $x$ in $\{0,...,k-1\}$, the notation $x^m$ means the concatenation of $m$ times $x$. 
For a word $w=w_0...w_n\in\{0,...,k-1\}^*$, we let $[w]_q$ denote the integer whose base-$q$ expansion is $w$. 
\begin{lem}\label{finite}
 Let $(u(n))_{n\geq 0}$ be a $k$-automatic sequence. Then the set of sequences 
 $$\{(u([1^n 0^j]_k))_{n\geq 1}\;\;|\:j\in \mathbb{N} \}$$
 is finite.
\end{lem}
\begin{proof}
 Let $M=(Q,\Sigma_k,\delta,q_0,\Delta,\tau)$ be a $k$-DFAO that generates $(u(n))_{n\geq 0}$. 
 Then $(u([1^n 0^j]_k))=\tau(\delta(q_0,1^n0^j))=\tau(\delta(\delta(q_0,0^j),1^n)).$ As $\delta(q_0,0^j)\in Q$ and $Q$ is finite, the set  
 $\{(u([1^n 0^j]_k))_{n\geq 0}\;\;|\:j\in \mathbb{N} \}$
 is finite.
\end{proof}

As in \cite{allouche}, we define 
$$\alpha=\prod\limits_{j=0}^{\infty}\left(1-\frac{T^{q^j}}{T^{q^{j+1}}}\right). $$
As $\alpha$ is algebraic over $\mathbb{F}_q(T)$, in order to prove Theorem \ref{main}, we only need to prove the transcendence of
$\frac{\alpha}{\Pi}L(1,\chi_s)$. From Theorem \ref{dm}, we deduce the expression that we will use for this article:
\begin{equation}\frac{\alpha}{\Pi}L(1,\chi_s)=\sum\limits_{k=0}^{\infty}(-1)^{k(s-1)}\left(\frac{1}{T}\right)^{(q-s)\frac{q^k-1}{q-1}}
\left(1-\frac{a}{T}\right)^{s\cdot\frac{q^k-1}{q-1}}\cdot \prod\limits_{j=k+1}^\infty\left(1-\left(\frac{1}{T}\right)^{q^j-1}\right).\tag{*}
\end{equation}

In Section \ref{proof}, we will prove the following proposition:

\begin{prop}\label{key}
 Let $s$ be an integer such that $1<s<q$. We denote by $u(n)$ coefficients of $\frac{1}{T^n}$ in $\frac{\alpha}{\Pi}L(1,\chi_s)$.
 Then for all $j\in \mathbb{N}$, the sequence $(u([1^n0^j]_q))_n$ is ultimately periodic and the length of the initial non-periodic segment
 of $(u([1^n0^j]_q))_n$ is a strictly increasing function with respect to $j$. In particular, the set
 $$\{(u([1^n 0^j]_q))_{n\geq 0}\;\;|\:j\in \mathbb{N} \}$$ is infinite.
\end{prop}

We obtain immediately the following Corollary using Theorem \ref{Christol} and Lemma \ref{finite}.
\begin{coro}
For $1<s<q$, series $\frac{\alpha}{\Pi}L(1,\chi_s)$ is transcendental over $\mathbb{F}_q(T)$.
\end{coro}

\section{Proof of Proposition \ref{key}}\label{proof}
We let $S_k$ denote the $k$-th summand in the expression $(*)$. First we observe that for $b\in \mathbb{N}$, the term $T^{-b}$ may 
appear in $S_k$ for more than one $k$. We want to determine $[T^{-b}]S_k$, the coefficient of $T^{-b}$ in $S_k$. For $1<s<q$, we denote $q-s$ by $\bar{s}$,
then from $(*)$ we see that if $[T^{-b}]S_k\neq 0$, then $b$ can be written as  
\begin{equation}
  b=r_k+\sum_{j=k+1}^{\infty}\varepsilon_j (q^j-1), \label{eq:dec}
\end{equation}
where $r_k\in [[\bar{s}^k]_q,[1^k0]_q ]$, $ \varepsilon_j\in\{0,1\}$  for $j\geq k+1  $ and $\varepsilon_j=0$ for $j$ big enough.
The following Lemma implies that such a decomposition is unique for $b$ and $k$.

\begin{lem}\label{ep}
i) Let $k$ and $l$ be positive integers such that $l\geq k$, then
$$ [1^k0]_q +\sum_{k+1\leq j \leq l} (q^j-1) < q^{l+1}-1.$$

ii) In particular, if $n$ can be written as $$ b=r_k+\sum_{j=k+1}^{\infty}\varepsilon_j (q^j-1), $$
where $r_k\in [[\bar{s}^k]_q,[1^k0]_q ]$, $ \varepsilon_j\in\{0,1\}$ not all $0$  and $\varepsilon_j=0$ for $j$ big enough, then
$$\max_{j\geq k+1}\{\varepsilon_j=1\}=  \max_{j\in\mathbb{N}}\{ q^j-1\leq b\}.$$
\end{lem}
\begin{proof}
i) \begin{align*}
  &[1^k0]_q +\sum_{k+1\leq j \leq l} (q^j-1)\\
  \leq &[1^k0]_q +\sum_{k+1\leq j \leq l} q^j\\
  =&[1^k0]_q +[1^{l-k}0^{k+1}]_q\\
  =&[1^l0]_q\\
  <&[1^{l+1}]_q
  \\\leq&q^{l+1}-1. 
 \end{align*}
ii) It is evident that $$J_1:=\max_{j\geq k+1}\{\varepsilon_j=1\}\leq  \max_{j\in\mathbb{N}}\{ q^j-1\leq b\}=:J_2.$$
Suppose that the inequality is strict. Then we would have
\begin{align*}
 b&=r_k+\sum_{j=k+1}^{J_1}\varepsilon_j (q^j-1)\\
&\leq [1^k0]+\sum_{j=k+1}^{J_1}(q^j-1)\\
&< q^{J_2}-1\\
&\leq b,
 \end{align*}
contradiction.

\end{proof}

For $b\in\mathbb{N}^*$, we can obtain all possible decompositions of $b$ of the form (\ref{eq:dec}) by applying repetitively Lemma \ref{ep}:

\begin{algorithm}[H]
\textbf{Input:} positive integer $b$\\
\textbf{Output:} finite sequence $(b)_n$ and a set $I$\\
$i:=1$;\\
$I:=\phi$\;
$b_1:=b$;\\
\If{$\exists l\in\mathbb{N}$ s.t. $[\bar{s}^l]_q\leq b_i\leq [1^l0]_q$}{
  add $i$ to $I$\;}
 \While{$\exists l\in\mathbb{N}^*$ s.t. $b_i\geq q^l-1$ }{
  $l_i:=\max\limits_{l\in\mathbb{N}^*}\{b\geq q^l-1\}$;\\

  \eIf{$b_i-(q^{l_i}-1)>[1^{l_i-1}0]_q$}{
   \textbf{end of procedure}\;
   }{
  $b_{i+1}:=b_i-(q^{l_i}-1)$\;
  $i++$\;
    \If{$\exists l\in\mathbb{N}$ s.t. $[\bar{s}^l]_q\leq b_i\leq [1^l0]_q$}{
  add $i$ to $I$\;}
  }
 }
\caption{Decomposition of $b$}
\end{algorithm}

Then all decompositions of $b$ in the form (\ref{eq:dec}) are $b_i+\sum\limits_{k=1}^{i-1}(b_k-b_{k+1})$ for $i\in I$.

As we are interested in the coefficients $u([1^m0^j]_q)$, we define $b_{j,m,1}=[1^m0^j]_q$ for $j,m\in\mathbb{N}^*$. And we define 
$b_{j,m,n}$ using
the procedure above with input $b_{j,m,1}$.

For example, for $j=2$ and $q=3$, the  base-$q$ expansion of $b_{j,m,n}$ is as follows, the symbol $*$ means that $b_{j,m,n}$ is not
defined:
\begin{center}
  \begin{tabular}{ l | c  r  r r r }
    \hline
     & 1 & 2 & 3 &4&5\\ \hline
    1 & 100 & 1 & *&* &*\\ 
    2 & 1100 & 101 & 2 &0 &*\\
    3&11100&1101&102& 10 &*\\
    4&111100&11101&1102&110&*\\
    \vdots&\vdots&\vdots&\vdots&\vdots&\vdots \\
   
  \end{tabular}
  \captionof{table}{$b_{2,m,n}$ for $q=3$}
\end{center}

We can observe some patterns from the table above, which we summarize in the following Lemma:

\begin{lem}\label{b}
For $j\geq 2$, the statement $P(n)$ is true for $1\leq n\leq q^{j-1}+1$ and the statements $Q(n)$ and $R(n)$ are true for $1\leq n\leq q^{j-1}$:\\
$P(n)$: For all $m\in\mathbb{N}^*$ and $m\geq n-1$, $b_{j,m,n}$ is defined and $b_{j,m+1,n}=b_{j,m,n}+q^{j+m+1-n}$.\\
$Q(n)$: For all $m\geq n$, $l_{j,m,n}:=\max\limits_{l\in\mathbb{N}^*}\{b_{j,m,n}\geq q^l -1\}=j+m-n$, 
and $b_{j,m,n}-(q^{l_{j,m,n}}-1)\leq [1^{l_{j,m,n}-1}0]_q$.
Thus $b_{j,m,n+1}=b_{j,m,n}-(q^{j+m-n}-1)$.\\
$R(n)$: For all $m\in\mathbb{N}^*$ and $m\geq n-1$, $b_{j,m+1,n+1}=b_{j,m,n}+1$.\\
\end{lem}
\begin{proof}
We prove by induction on $n$.

For $n=1$, $P(1)$ is true by definition of $b_{j,m,1}$.

To prove $Q(1)$ we use induction on $m$. First, 
$$l_{j,1,1}=\max\limits_{l\in\mathbb{N}^*} \{b_{j,1,1}\geq q^l-1\}=\max\limits_{l\in\mathbb{N}^*}\{[10^j]_1\geq q^l-1\}=j+1-1.$$
And 
$$b_{j,1,1}-(q^{l_{j,1,1}}-1)=1\leq [1^{l_{j,1,1}-1}0]_q.$$
Suppose that the statements are true for $m$, using $P(1)$ we have
$$l_{j,m+1,1}=\max\limits_{l\in\mathbb{N}^*} \{b_{j,m+1,1}\geq q^l-1\}=\max\limits_{l\in\mathbb{N}^*} \{b_{j,m,1}+q^{j+m+1-1}\geq q^l-1\}
=l_{j,m,1}+1=j+m+1-1,$$
and
\begin{align*}
&b_{j,m+1,1}-(q^{j+m+1-1}-1)\\
=&(b_{j,m,1}+q^{j+m+1-1})-q^{j+m+1-1}+1\\
=&b_{j,m,1}+1\\
=&b_{j,m,1}-(q^{j+m-1}-1)+q^{j+m-1}\\
\leq &[1^{j+m-1-1}0]_q+q^{j+m-1}\\
=&[1^{j+m-1}0]_q,
 \end{align*}
which proves $Q(1)$.

From $P(1)$ and $Q(1)$ follows $R(1)$.

Suppose that for $n<q^{j-1}$, we have proven $P(n')$, $Q(n')$ and $R(n')$ for all $n'\in\{1,...,n\}$. Let us prove $P(n+1)$, $Q(n+1)$ and 
$R(n+1)$.

First, $P(n+1)$ can be deduced immediately from $P(n)$ and $R(n)$.

For $Q(n+1)$, we prove by induction on $m\geq n+1$. By $R(1),...,R(n)$ we have $b_{j,n+1,n+1}=b_{j,1,1}+n\leq q^j+q^{j-1}-1$. Therefore
$$l_{j,n+1,n+1}=\max\limits_{l\in\mathbb{N}^*} \{b_{j,n+1,n+1}\geq q^l-1\}\leq
\max\limits_{l\in\mathbb{N}^*} \{ q^j+q^{j-1}-1\geq q^l-1\}=j,$$
on the other hand,
$$l_{j,n+1,n+1}=\max\limits_{l\in\mathbb{N}^*} \{b_{j,n+1,n+1}\geq q^l-1\}\geq 
\max\limits_{l\in\mathbb{N}^*} \{q^j\geq q^l-1\}=j.$$
Therefore $l_{j,n+1,n+1}=j=j+(n+1)-(n+1)$.
Besides,
$$b_{j,n+1,n+1}-(q^{l_{j,n+1,n+1}}-1)\leq q^j+q^{j-1}-1-(q^j-1)=q^{j-1}\leq [1^{j-1}0]_q,$$
which proves $Q(n+1)$.
From $P(n+1)$ and $Q(n+1)$ follows $R(n+1)$.

Finally, $P(q^{j-1}+1)$ can be deduced from $Q(q^{j-1})$ and $R(q^{j-1})$.
\end{proof}

\begin{coro}\label{L:s}
For $j\geq 2$, $1\leq n\leq q^{j-1}+1$ and $m\geq n$, $b_{j,m,n}\in \{[\bar{s}^{j+m-n}]_q,...,[1^{j+m-n}0]_q\}$.\\
\end{coro}

Now we look at the table of $b_{j,m,n}$ for $j=4$ and $q=3$.

\textcolor{red}
\tiny{}
\begin{center}
\tiny{}
  \begin{tabular}{ l | c  r  r r r r r r r r r r r r}
    \hline
     & 1 & 2 & 3 &4&5&6&7&8&9&10&11&12&13&14\\ \hline
    1 & 1000 & 1 & *&* &*&*&*&*&*&*&*&*&*&*\\ 
    2 & 11000 & 1001 & 2 &0 &*&*&*&*&*&*&*&*&*&*\\
    3&111000  &11001 &1002& 10 &1&*&*&*&*&*&*&*&*&*\\
    4& .  &111001&11002&1010&11&2&*&*&*&*&*&*&*&*\\
    5& . & .  &111002&11010&1011&12&10&*&*&* &*&*&*&*\\
    6& . & . & . &111010&11011&1012&20&11&*&*&*&*&*&*\\
    7&.  &.  &.  &.     &111011&11012&1020&21&12&*&*&*&*&*\\
    8&.  &.  &.  &.     &.     &111012&11020&1021&22&0&*&*&*&*\\
    9&.  &.  &.  &.     &.     &.     &111020&11021&1022&100&1&*&*&*\\
    10&.  &.  &.  &.     &.     &.     &.     &111021&11022&1100&101&2&0&*\\
    11&.  &.  &.  &.     &.     &.     &.     &.     &111022&11100&1101&102&10&*\\
    12&.  &.  &.  &.     &.     &.     &.     &.     &     .&111100&11101&1102&110&*\\
    13&.  &.  &.  &.     &.     &.     &.     &.     &     .&.&111101&11102&1110&*\\
  \end{tabular}
  \captionof{table}{$b_{3,m,n}$ for $q=3$}
\end{center}

\normalsize{}
We notice that starting from $m=9$ and $n=10$, the subtable is the same as that of $j=2$ and $q=3$. It is the case in general that the table of $b_{j,m,n}$ occurs at the end of the table of $b_{j+1,m,n}$.
\begin{lem}\label{L:end}
 For $j,m,n\in\mathbb{N}^*$, $n\geq q^j+1$ and $m\geq q^j$,
$b_{j+1,m,n}$ is defined if and only if $b_{j,m-q^j+1,n-q^j}$ is defined. When they are defined they have the same value.
\end{lem}
\begin{proof}
By the definition of $b_{j,m,n}$, the first two columns determine the rest of the table. Therefore we only need to prove that
for all $m\geq q^j$,
\begin{equation} b_{j+1,m,q^j+1}=b_{j,m-q^j+1,1}\label{eq:1}\end{equation}
and                                                
\begin{equation}b_{j+1,m,q^j+2}=b_{j,m-q^j+1,2}. \label{eq:2}\end{equation}
By applying Lemma \ref{b} we have
$$b_{j+1,q^j,q^j+1}=b_{j+1, 1,2}+q^j-1=q^j=b_{j,1,1},$$
and for $k\in\mathbb{N}$,
$$b_{j+1, q^j+k+1, q^j+1}-b_{j+1,q^j+k,q^j+1}=q^{j+k+2}.$$
Thus for $m\geq q^j $ $$b_{j+1,m,q^j+1}=[1^{m+1-q^j}0^j]_q=b_{j,m+1-q^j,1}$$
which proves (\ref{eq:1}).

For $m\geq q^j$, accroding to the second point Lemma \ref{b},
$$b_{j+1,m,q^j+1}-b_{j+1,m,q^j+2}=q^{j+1+m-(q^j+1)}-1=q^{j+(m-q^j+1)-1}=b_{j,m-q^j+1,1}-b_{j,m-q^j+1,2}.$$
which proves (\ref{eq:2}).
\end{proof}

To calculate the coefficient of $T^{-n}$ in $S_k$, we define $k_{j,m,n}$ and $c_{j,m,n}$ as follows:
When $b_{j,m,n}$ is defined and there exists $k\in\mathbb{N}$ such that $b_{j,m,n}\in \{[\bar{s}^k]_q,...,[1^k0]_q\}$, 
$k_{j,m,n}$ is defined to be $k$. Otherwise $k_{j,m,n}$ is not defined. When $k_{j,m,n}$ is defined, $c_{j,m,n}$ is defined to be
$b_{j,m,n}-[\bar{s}^{k_{j,m,n}}]_q$.
Finally we define $N_{j,m}$ to be $\{n\in \mathbb{N}^* \mbox{ such that } c_{j,m,n} \mbox{ is defined}\}$.
From the expression ($*$) we see that:

\begin{lem}
 For $j,m\in \mathbb{N}^*$, 
 $$\left[T^{[1^m0^j]_q}\right]\frac{\alpha}{\Pi}L(1,\chi_s)=\sum_{n\in N_{j,m}} (-1)^{k_{j,m,n}(s-1)}\binom{\left[s^{k_{j,m,n}}\right]_q}
 {c_{j,m,n}} (-a)^ {c_{j,m,n}} \cdot(-1)^{n-1}.$$
\end{lem}
For $n\in N_{j,m}$, we denote by $d_{j,m,n}$ the quantity $(-1)^{k_{j,m,n}(s-1)}\binom{\left[s^{k_{j,m,n}}\right]_q}
 {c_{j,m,n}} (-a)^ {c_{j,m,n}} \cdot(-1)^{n-1}$. For other $n\in\mathbb{N}^*$, we define $d_{j,m,n}$ to be $0$ for convenience.

In order to calculate the coefficients we need the following Theorem:

\begin{thm}[Lucas]\label{lucas}
 Let $p$ be a prime number and $q=p^k$ for $k\in\mathbb{N}^*$. Let $m=\sum\limits_i m_i q^i$, $n=\sum\limits_j n_j q^j$ be two integers, where 
 $m_i, n_j\in \{0,1,...,q-1\}$. Then
 $$\binom{m}{n}\equiv \prod\limits_{i}\binom{m_i}{n_i}\;\mod p.$$
\end{thm}

Let us look at an example of $c_{j,m,n}$ and $d_{j,m,n}$ with $j=2$, $q=3$ and $s=2$:

\begin{center}
  \begin{tabular}{ l | c  r  r r r }
    \hline
     & 1 & 2 & 3 &4&5\\ \hline
    1 & 12 & 0 & *&* &*\\ 
    2 & 212 & 20 & 1 &0&*\\
    3&2212&220&  21   & 2 &*\\
    4&22212&2220&221&22&*\\
    \vdots&\vdots&\vdots&\vdots&\vdots&\vdots \\
   
  \end{tabular}
  \captionof{table}{$c_{2,m,n}$ for $q=3$ and $s=2$}
\end{center}

\begin{center}
  \begin{tabular}{ l | c  r  r r r }
    \hline
     & 1 & 2 & 3 &4&5\\ \hline
    1 & $2(-a)^3$ & 1 & 0&0 &0\\ 
    2 & $-2(-a)^5$ & $-(-a)^2$ & $-2(-a)$ &-1&0\\
    3&$2(-a)^7$&$(-a)^4$&  $2(-a)^3$   & $(-a)^2$ &0\\
    4&$-2(-a)^9$&$-(-a)^6$&$-2(-a)^5$&$-(-a)^4$&0\\
     5&$2(-a)^{11}$&$(-a)^8$&$2(-a)^7$&$(-a)^6$&0\\
    \vdots&\vdots&\vdots&\vdots&\vdots&\vdots \\
   
  \end{tabular}
  \captionof{table}{$d_{2,m,n}$ for $q=3$ and $s=2$}
\end{center}

From the table we observe that $(d_{j,m,n})_{m\geq n}$ seems to be periodic. Indeed, we have:
\begin{lem}\label{periodic}
 For $j\in\mathbb{N}^*$ and $1\leq n\leq q^{j-1}+1$, the sequence $(d_{j,m,n})_{m\geq n}$ is periodic.
\end{lem}
\begin{proof}
Throughout this proof we suppose that $1\leq n\leq q^{j-1}+1$ and $m\geq n$. 

 First, we know from Corollary \ref{L:s} that $k_{j,m,n}$ and thus $d_{j,m,n}$, are defined and $k_{j,m,n}=j+m-n$. 
 From Lemma \ref{b} we see that 
 $$0<c_{j,m,n}\leq q^{j+m-n}-1$$
 and
 $$c_{j,m+1,n}=c_{j,m,n}+s\cdot q^{j+m-n}.$$
 Therefore 
\begin{align*}
 d_{j,m+1,n}&=(-1)^{k_{j,m+1,n}(s-1)}\binom{\left[s^{k_{j,m+1,n}}\right]_q}{c_{j,m+1,n}} (-a)^ {c_{j,m+1,n}} \cdot(-1)^{n-1}\\
 &=(-1)^{(j+m+1-n)(s-1)}\binom{\left[s^{j+m+1-n}\right]_q}
 {s\cdot q^{j+m-n}+c_{j,m,n}} (-a)^ {s\cdot q^{j+m-n} +c_{j,m,n}} \cdot(-1)^{n-1}\\
 &=(-1)^{s-1}(-1)^{k_{j,m,n}(s-1)}\binom{s}{s} \binom{\left[s^{k_{j,m,n}}\right]_q}{c_{j,m,n}} (-a)^{s\cdot q^{j+m-n}} (-a)^ {c_{j,m,n}}
 \cdot (-1)^{n-1}\\
 &=(-1)^{s-1}(-a)^s \cdot d_{j,m,n}.
\end{align*}
As $(-1)^{s-1}(-a)^s$ is an element in a finite field, if $a\neq 0$, the sequence $(d_{j,m,n})_{m\geq n}$ is periodic. If $a=0$, as 
$c_{j,n,n}\neq 0$, the sequence $(d_{j,m,n})_{m\geq n}$ is always $0$, therefore also periodic.
\end{proof}

For an ultimately periodic sequence $(a_n)_n$ we define $IN((a_n)_n)$ to be the index of the 
earlist term from which the sequence
is periodic.
That is, $$IN((a_n)_n)=\min\limits_{i}\{(a_n)_{n\geq i} \mbox{ is periodic}\}.$$

The idea of the proof of Proposition \ref{key} is that the sequences  $(d_{j,m,n})_{m\geq 1}$ are ultimately periodic and the
$IN((d_{j,m,n})_{m\geq 1})$ increases with $n$ for $n\leq n_0:=\sum\limits_{i=0}^{j-1}q^i$. 
For $n>n_0$, the sequence $(d_{j,m,n})_{m\geq 1}$ is zero. 
We have $u\left([1^m0^j]_q\right)=\sum\limits_{n= 1}^{n_0} d_{j,m,n}$ 
and $IN\left(\left(u\left([1^m0^j]_q\right)\right)_{m\geq 1}\right)$ is not far from 
the $IN\left((d_{j,m, n_0})_{m\geq 0}\right)$. In order to justify the last point, 
we need to take a closer look at the table of $d_{2,m,n}$, which according to Lemma \ref{L:end} occurs at the end of 
the table of $d_{j,m,n}$ for $j\geq 3$. 

From the proof of Lemma \ref{periodic} and the definition of $b_{j,m,n}$ and $d_{j,m,n}$ it is easy to give an explicit 
expression of $d_{2,m,n}$:

\begin{lem}\label{L:j=2}
For $n=1$, $$d_{2,m,n}=\binom{s}{s-1}(-a)^{(s-1)+(m+1-n)\cdot s}(-1)^{(s-1)(m-n)}(-1)^{n-1}.$$ \\
For $2 \leq n\leq \bar{s}$, $d_{2,m,n}=0$ for all $m\in\mathbb{N}^*$. \\
For $\bar{s}+1\leq n \leq q$,
$$d_{2,m,n}=\begin{cases} 0&\mbox{if } m<n-1\\
                         \binom{s}{n-1-\bar{s}}(-a)^{(n-1-\bar{s})+(m+1-n)\cdot s} (-1)^{(s-1)(m-n)}(-1)^{n-1} & \mbox{if } m\geq n-1
            \end{cases}$$            
For $n=q+1$,            
$$d_{2,m,n}=\begin{cases} 0&\mbox{if } m<n-2\\
                         (-a)^{(m+2-n)\cdot s} (-1)^{(s-1)(m-n)}(-1)^{n-1} & \mbox{if } m\geq n-2
            \end{cases}$$
For $n> q+1$, $d_{2,m,n}=0$ for all $m\in\mathbb{N}^*$ 
 
\end{lem}

From the table of $d_{2,m,n}$ and $d_{3,m,n}$ we can see easily that the following Lemma is true. We provide nontheless a 
proof for the sake of completeness.

\begin{lem}\label{L:fin}
 $IN\left( \left(u\left([1^m0^3]_q\right)\right)_m \right)\geq q^{2}$.
\end{lem}

\begin{proof}
We divide the argument into two cases. 

Case 1: $1-s\cdot a^{s-1}\neq 0$ or $q>3$.

In this case we prove that $IN\left( \left(u\left([1^m0^2]_q\right)\right)_m \right)\geq 2.$ Thus by Lemma \ref{L:end} and Lemma
\ref{L:j=2} we have $IN\left( \left(u\left([1^m0^3]_q\right)\right)_m \right)\geq q^{2}$.

Case 1.1: If $a=0$, then from Lemma \ref{L:j=2} we know that 
$$\max\limits_{m}\{\exists n\in\mathbb{N}^* \mbox{ s. t. } d_{2,m,n}\neq 0\}=q-1.$$ 
Therefore  $IN\left( \left(u\left([1^m0^2]_q\right)\right)_m \right)= q-1$.

Case 1.2
If $a\neq 0$ and $1-s\cdot a^{s-1}\neq 0$, we rewrite the expressions in Lemma \ref{L:j=2} for $n=q$ and $n=q+1$ as
\begin{align*}
d_{2,m,q}=&\begin{cases} 0&\mbox{if } m<q-1\\
                         s\cdot (-a)^{(s-1)+(m+1-q)\cdot s} (-1)^{(s-1)(m-q)}(-1)^{q-1} & \mbox{if } m\geq q-1.\\
            \end{cases}\\
d_{2,m,q+1}=&\begin{cases} 0 &\mbox{if } m<q-1\\
                         (-a)^{(m+1-q)\cdot s} (-1)^{(s-1)(m-q-1)}(-1)^{q} & \mbox{if } m\geq q-1.
            \end{cases}
\end{align*}
Therefore
$$d_{2,m,q}+d_{2,m,q+1}=\begin{cases} 0 &\mbox{if } m<q-1\\
                         (1-s\cdot a^{s-1})(-a)^{(m+1-q)\cdot s} (-1)^{(s-1)(m-q-1)}(-1)^{q} & \mbox{if } m\geq q-1.
            \end{cases}$$
Since $1-s\cdot a^{s-1}\neq 0$, $IN((d_{2,m,q}+d_{2,m,q+1})_m)=q-1$. By Lemma \ref{L:j=2}, $IN((d_{2,m,n})_m)\leq q-2$
for all $1\leq n\leq q-1$. Therefore  $IN\left( \left(u\left([1^m0^2]_q\right)\right)_m \right)=q-1\geq 2$.

Case 1.3
If $1-s\cdot a^{s-1}= 0$ and $q\geq 4$, $IN((d_{2,m,q}+d_{2,m,q+1})_m)=1$. Since $s>1$, $\bar{s}+1\leq q-1$. Therefore by
Lemma \ref{L:j=2} we have $IN(d_{2,m,q-1})=q-2>1$ since $q\geq 4$. 
By Lemma \ref{L:j=2}, for $1\leq n<q-1$, $IN(d_{2,m,n})\geq \max\{1,...,q-3\}$. 
Therefore $IN\left( \left(u\left([1^m0^2]_q\right)\right)_m \right)=q-2\geq 2$.

Case 2: If $1-s\cdot a^{s-1}=0$ and $q=3$, using the formula in Lemma \ref{L:j=2} we find that
$d_{2,m,n}=0$ for $m,n\in\mathbb{N}^*$.
So we look at the table of $d_{3,m,n}$. With similar calculation we find for $n\leq q^2$, the sequences $(d_{3,m,n})_m$ 
are periodic from $m=q^2-1$. But the sum of the last four columns ($n=q^2+1,...,q^2+4$) is only periodic from $m=q^2$. 
Therefore $IN\left( \left(u\left([1^m0^3]_q\right)\right)_m \right)=q^2$.

\end{proof}

\begin{proof}[Proof of Proposition \ref{key}]
We prove by induction on $j$ that $IN\left( \left(u\left([1^m0^j]_q\right)\right)_m \right)\geq q^{j-1}$ for $j\geq 3$. 
By Lemma \ref{L:fin} we know that $IN\left( \left(u\left([1^m0^3]_q\right)\right)_m \right)\geq q^{2}$. 

Suppose that for we have proven for $j$ that $IN\left( \left(u\left([1^m0^j]_q\right)\right)_m \right)\geq q^{j-1}$.
We define $n_0:=\sum\limits_{i=0}^{j}q^i$ and $n_1=\sum\limits_{i=0}^{j-1}q^i$. Then 
$$u\left([1^m0^{j+1}]_q\right)=\sum\limits_{n=1}^{n_0}d_{j+1,m,n}=\sum\limits_{n=1}^{q^j}d_{j+1,m,n}+
\sum\limits_{n=1}^{n_1}d_{j+1,m,n+q^j}.$$
By Lemma \ref{periodic} we know that $IN(\sum\limits_{n=1}^{q^j}d_{j+1,m,n})\leq q^j$.
By Lemma \ref{L:end} we know that 
$$\sum\limits_{n=1}^{n_1}d_{j+1,m+q^j-1,n+q^j}=\sum\limits_{n=1}^{n_1}d_{j,m,n}\cdot(-1)^q=u\left([1^m0^j]_q\right)\cdot(-1)^q.$$
 By the hypothesis of induction we have
 $$IN((\sum\limits_{n=1}^{n_1}d_{j+1,m,n+q^j})_m)= q^j-1+IN((u([1^m0^j]_q))_m)\geq q^j+2.$$
When we have two ultimately periodic sequences $u$ and $v$ such that $IN(u)>IN(v)$, we have $IN(u+v)=IN(u)$. Therefore
$$IN(u\left([1^m0^{j+1}]_q\right))=IN(\sum\limits_{n=1}^{n_1}d_{j+1,m,n+q^j})\geq q^j+2.$$
This completes the proof.
\end{proof}

\end{document}